\documentclass[reqno]{amsart}
\usepackage{amsmath,amsthm,amssymb,hyperref,graphicx,mathrsfs,mathtools} 

\newtheorem{theorem}{Theorem}
\newtheorem{corollary}{Corollary}
\newtheorem{lemma}{Lemma}

\newtheorem{remark}{Remark}

\allowdisplaybreaks
\title[Bernstein-type $L_p$ inequalities]{Certain Bernstein-type $L_p$ inequalities for polynomials}

\begin{document}
\author{N. A. Rather$^1$}
\author{Aijaz Bhat$^2$}
\author{Suhail Gulzar$^3$}

\address{$^{1,2}$Department of Mathematics, University of Kashmir, Srinagar-190006, India}
\address{$^3$Government College of Engineering and Technology, Kashmir}
\email{$^1$dr.narather@gmail.com, $^2$aijazb824@gmail.com,$^3$Suhail Gulzar}

\begin{abstract}
Let $P(z)$ be a polynomial of degree $n,$ then it is known that for $\alpha\in\mathbb{C}$ with $|\alpha|\leq \frac{n}{2},$ 
\begin{align*}
\underset{|z|=1}{\max}|\left|zP^{\prime}(z)-\alpha P(z)\right|\leq \left|n-\alpha\right|\underset{|z|=1}{\max}|P(z)|.
\end{align*}
This inequality includes Bernstein's inequality, concerning the estimate for $|P^\prime(z)|$ over $|z|\leq 1,$ as a special case. In this paper, we extend this inequality to $L_p$ norm which among other things shows that the condition on $\alpha$ can be relaxed. We also prove similar inequalities for polynomials with restricted zeros. 
\end{abstract}
\maketitle


\section{Introduction and Main Results}
Let $\mathcal P_{n}$ denote the set of complex polynomials of degree  $n.$ If $P\in\mathcal{P}_n,$ then according to Bernstein's inequality concerning the estimate of $|P(z)|$ over $|z|\leq 1,$ we have
\begin{align*}
\left\|P^\prime(z)\right\|_{\infty} \leq n \left\|P(z)\right\|_{\infty}, \quad \text{where}\,\,\, \left\|P\right\|_{\infty}:=\underset{\left|z \right|=1}{\max}\left|P(z)\right|.
\end{align*}
The inequality is sharp and becomes equality if $P(z)=c~z^n,$ $c\neq 0.$ This inequality of Bernstein has been generalized and extended in several ways, for more details refer to \cite[p. 508]{QIGS}. Under the same hypothesis, V.K. Jain \cite{jain} proved that if $\alpha\in\mathbb{C}$ with $|\alpha|\leq \frac{n}{2},$ then 
\begin{align}\label{je}
\left\|zP^{\prime}(z)-\alpha P(z)\right\|_\infty \leq \left|n-\alpha\right|\left\|P(z)\right\|_\infty.
\end{align}
Zygmund \cite{zyg} extended Bernstein's inequality in $L_p$-norm and proved that if $P\in\mathcal{P}_n,$   then
\begin{equation}\label{ie1}
\Vert P^{\prime}(z)\Vert_{p}\leq n\Vert P(z)\Vert_{p}\quad \text{for} \,\,\, p\geq 1
\end{equation}
where
\[\left\|P(z)\right\|_{p}:=\left(\frac{1}{2\pi}\int_{0}^{2\pi}\left|
P(e^{i\theta})\right|^{p}d\theta\right)^{1/p},\, 0<p<\infty.\]
It is well known that the supremum norm satisfies $\left\|P\right\|_{\infty}=\underset{p\to\infty}{\lim}\left\|P\right\|_{p}  .$ The Mahler measure (see \cite{mah1}) of a polynomial $P(z)$ with complex coefficients is defined as
$\left\|P(z)\right\|_{0}:=\exp\left(\frac{1}{2\pi}\int_{0}^{2\pi}\log\left|P(e^{i\theta})\right|d\theta\right).$ This also satisfy the limiting case $\underset{p\to 0}{\lim}\left\|P\right\|_{p}=\|P\|_0.$

It was not known for quite a long time that whether the inequality \eqref{ie1} is true for $0<p<1.$ or not. Finally, it was Arestov \cite{ART} who came up with some remarkable results which among other things proved that the inequality \eqref{ie1} remains valid for
$0\leq p<1$ as well. 

Recently in \cite{sg}, the inequality \eqref{je} of Jain was extended to $L_p$-norm and it was proved that if $P\in\mathcal{P}_n$ and $\alpha\in\mathbb{C}$ with $|\alpha|\leq n/2,$ then
\begin{align}\label{ie2}
\left\|zP^{\prime}(z)-\alpha P(z)\right\|_p \leq \left|n-\alpha\right|\left\|P(z)\right\|_p\quad \text{for}\,\, p\geq 0.
\end{align}
It is not clear from the proofs of \eqref{je} and \eqref{ie2} that whether the condition on $\alpha$ is indeed necessary. In this paper, we have been able to relax the condition on $\alpha$ and have proved the following theorem.
\begin{theorem}\label{t1}
If $P\in \mathcal{P}_n$,  then for every $\alpha,\gamma \in\mathbb{C}$ with $\Re(\alpha)\leq \frac{n}{2}$, $\Re(\gamma)\leq \frac{n}{2}$  and $0\leq p<\infty$
\begin{align}\label{tlel}
\left\|zP^\prime(z)-\alpha P(z)\right\|_p\leq \left|n-\alpha\right|\left\|P(z)\right\|_p
\end{align}
and
\begin{align}\label{t1ec}
\left\|z^2P^{\prime\prime}(z)-(1-\alpha-\gamma )zP^\prime(z)+\alpha\gamma P(z)\right\|_p\leq \left|n(n-\alpha-\gamma )+\alpha\gamma\right|\left\|P(z)\right\|_p.
\end{align}
The results are best possible and equality in \eqref{tlel} and \eqref{t1ec} holds for $P(z)=cz^n, ~ c \neq 0.$
\end{theorem}
\begin{remark}
Consider $P(z)=2z^n+1$ and $p=2,$ then
\begin{align*}
\left\|zP^\prime(z)-\alpha P(z)\right\|_p=\left\|2(n-\alpha)z^n-\alpha\right\|_2=\sqrt{4|n-\alpha|^2+|\alpha|^2}
\end{align*}
and 
\begin{align*}
\left|n-\alpha\right|\left\|P(z)\right\|_p&=|n-\alpha|\|2z^n+1\|_2=\sqrt{5}\left|n-\alpha\right|.
\end{align*}
Thus,
\begin{align*}
&\left\|zP^\prime(z)-\alpha P(z)\right\|_p\leq \left|n-\alpha\right|\left\|P(z)\right\|_p\\  \Leftrightarrow &~~4|n-\alpha|^2+|\alpha|^2 \leq 5|n-\alpha|^2 \Leftrightarrow |\alpha|\leq |n-\alpha| \Leftrightarrow \Re(\alpha)\leq \frac{n}{2}.
\end{align*}
This shows that the condition on $\alpha$ in \eqref{tlel} cannot be relaxed furthermore.
\end{remark}
\begin{remark}
If we let $p\to\infty,$ one can conclude that the inequality \eqref{je} holds for all $\alpha\in\mathbb{C}$ with $\Re(\alpha)\leq\frac{n}{2}.$ If we take $\alpha=0$ in \eqref{tlel} we obtain Zygmund's inequality \eqref{ie1}.
\end{remark}
 Taking $\alpha =\gamma =\frac{n}{2}$ in \eqref{tlel} and \eqref{t1ec}, we get the following result:
 \begin{corollary}\label{c1}
 Let $P\in\mathcal{P}_n,$ then
 \begin{align*}
 \left\|zP^\prime(z)-\frac{n}{2} P(z)\right\|_p\leq \frac{n}{2}\left\|P(z)\right\|_p
 \end{align*}
 and
 \begin{align*}
 \left\|z^2P^{\prime\prime}(z)+(1-n)zP^\prime(z)-\frac{n^2}{4} P(z)\right\|_p\leq \frac{n^2}{4}\left\|P(z)\right\|_p.
 \end{align*}
 The results are sharp and equality holds for $P(z)=cz^n, ~ c \neq 0.$
 \end{corollary}
The next corollary follows by taking $\alpha =1$ in \eqref{tlel} and $\gamma=0$  in \eqref{t1ec}.
\begin{corollary}\label{c2}
 Let $P\in\mathcal{P}_n,$ then
 \begin{align*}
 \left\|zP^\prime(z)- P(z)\right\|_p\leq \left|n-1\right|\left\|P(z)\right\|_p
 \end{align*}
 and
 \begin{align*}
 \left\|z^2P^{\prime\prime}(z)+(1-\alpha)zP^\prime(z)\right\|_p\leq n\left|n-\alpha\right|\left\|P(z)\right\|_p \quad (\textnormal{for} \,\, n\geq 2).
 \end{align*}
 The results are best possible and equality holds for $P(z)=cz^n, ~ c \neq 0.$
 \end{corollary}
 If $P\in \mathcal{P}_n$ and does not vanish in $|z|\leq 1,$ then it was conjectured by Erd\"{o}s and later proved by Lax \cite{lax} that the factor $n$ in the right side of Bernstein's inequality can be replaced by $n/2.$ Jain \cite{jain} also generalized this inequality and proved that if $\alpha\in\mathbb{C}$ with $|\alpha|\leq \frac{n}{2},$ then 
\begin{align}\label{je2}
\left\|zP^{\prime}(z)-\alpha P(z)\right\|_\infty\leq \dfrac{\left|n-\alpha\right|+\left|\alpha\right|}{2}\left\|P(z)\right\|_\infty.
\end{align}
 Next, we present the following result which not only includes an $L_p$-norm extension of above inequality but also extends the range of $\alpha.$
 \begin{theorem}\label{t2}
If $P\in \mathcal{P}_n$ and $P(z)$ does not vanish in $|z|\leq 1$, then for every $\alpha,\gamma\in\mathbb{C}$ with $\Re(\alpha)\leq \frac{n}{2}$, $\Re(\gamma)\leq \frac{n}{2}$ and $0\leq p<\infty$ 
\begin{align}\label{cel} 
\left\|zP^\prime(z)-\alpha P(z)\right\|_p \leq \left\|\left(n-\alpha\right)z+\alpha \right\|_p \frac{\left\|P(z)\right\|_p}{\left\|1+z\right\|_p}
\end{align} 
and
\begin{align}\label{sss}
\left\|z^2P^{\prime\prime}(z)-(1-\alpha-\gamma )zP^\prime(z)+\alpha\gamma P(z)\right\|_p\leq \left\|n(n-\alpha-\gamma )z+\alpha\gamma\right\|_p\frac{\left\|P(z)\right\|_p}{\left\|1+z\right\|_p}.
\end{align}
 The results are best possible and equality holds for $P(z)=az^n+b, ~|a|=|b|\neq 0$.
\end{theorem} 
\begin{remark}
Inequality \eqref{je2} follows from \eqref{cel} by letting $p\to\infty,$ which also shows that \eqref{je2} holds for $\Re(\alpha)\leq \frac{n}{2}.$ If we take $\alpha=0$ in \eqref{cel} , it reduces to $L_p$-norm extension of Erd\"{o}s-Lax inequality \cite{lax} due to de Bruijn \cite{DB}.
\end{remark}
The following Corollary is obtained from Theorem \ref{t2} by taking $\alpha =1$ in \eqref{cel} and $\gamma=0$ in \eqref{sss}.
\begin{corollary}
If $P\in \mathcal{P}_n$ and $P(z)$ does not vanish in $|z|\leq 1,$ then
\begin{align*}
\left\|zP^\prime(z)- P(z)\right\|_p \leq \left\|\left(n-1\right)z+1 \right\|_p \frac{\left\|P(z)\right\|_p}{\left\|1+z\right\|_p}.
\end{align*} 
and
\begin{align*}
\left\|z^2P^{\prime\prime}(z)-(1-\alpha )zP^\prime(z)\right\|_p\leq n\left|n-\alpha \right|\frac{\left\|P(z)\right\|_p}{\left\|1+z\right\|_p}.
\end{align*}
\end{corollary} 
A polynomial $P\in\mathcal{P}_n$ is said to be \textit{self-inversive} polynomial if $P(z) = uz^n \overline{P\big({1}/{\overline{z}}\big)},$ where $u\in\mathbb{C}$ with $|u|=1.$

Finally, we show that the Theorem \ref{t2} also holds true for the class of self-inversive polynomials. 
\begin{theorem}\label{t3}
Inequalities \eqref{cel} and \eqref{sss} also holds if   $P(z)$ is a self-inversive polynomial. 
 The inequalities are best possible and equality holds for $P(z)=az^n+\bar{a}, ~|a|\neq 0$.
\end{theorem}
 \section{Lemmas}
For the proof of these theorems, we need the following lemmas. The first lemma is the following well known-result ([\cite{QIGS} Theorem 14.1.2 and its proof , corollary 12.1.3]. Also [\cite{DB}, Theorem 1 and its proof]).
\begin{lemma}\label{l1}
Let  $F\in \mathcal{P}_n$ and $P(z)$ be a polynomial of degree at most $n$ such that $|P(z)|\leq |F(z)|$ for $|z|=1$. If $F(z)\neq0$ for $|z|<1$ (respectively $|z|>1$) and for every $z\in \mathbb{C}$ and $\delta\in\mathbb{R},$  $ P(z)\neq e^{i\delta}F(z),$ then
\begin{itemize}
\item[(i)]  $|P(z)|< |F(z)|$ for $|z|<1$ (respectively $|z|>1$),
\item[(ii)] $F(z)+\beta P(z) \neq 0$ for $|z|<1$ (respectively $|z|>1$) and $\beta \in \mathbb{C}$ with $|\beta| \leq 1$ and 
\item[(iii)] $P(z)+\lambda F(z) \neq 0$ for $|z|<1$ (respectively $|z|>1$) and $\lambda \in \mathbb{C}$ with $|\lambda| \geq 1.$ 
\end{itemize}
\end{lemma}
\begin{lemma}\label{l2}
Let all the zeros of an $n$th degree polynomial $f(z)$ lie in $|z|\leq r,$ then for every $\alpha\in\mathbb{C}$ with $\Re(\alpha)\leq \frac{n}{2},$ the zeros of $zf^\prime(z)-\alpha f(z)$ also lie in $|z|\leq r.$
\end{lemma}
\begin{proof}
Let $\psi(z)=zf^\prime(z)-\alpha f(z)$ and $w\in\mathbb{C}$ with $|w|>r.$ Suppose $z_1,z_2,\ldots,z_n$ be the zeros of $f(z),$ then $|z_\nu|\leq r$ and $|w|-|z_\nu|>0$ for $\nu=1,2,\ldots,n.$ Now,
\begin{align*}
\dfrac{\psi(w)}{f(w)}&=\dfrac{wf^\prime(w)}{f(w)}-\alpha=\sum_{\nu=1}^{n}\dfrac{w}{w-z_\nu}-\alpha\\&=\dfrac{1}{2}\sum_{\nu=1}^{n}\dfrac{(w-z_\nu)+(w+z_\nu)}{w-z_\nu}-\alpha\\&=\dfrac{n}{2}-\alpha+\dfrac{1}{2}\sum_{\nu=1}^{n}\dfrac{(w+z_\nu)(\bar{w}-\bar{z_\nu})}{|w-z_\nu|^2}.
\end{align*}
This implies
\begin{align*}
\Re\left( \dfrac{\psi(w)}{f(w)} \right)&=\dfrac{n}{2}-\Re(\alpha)+\dfrac{1}{2}\sum_{\nu=1}^{n}\dfrac{|w|^2-|z_\nu|^2}{|w-z_\nu|^2}\\&\geq \dfrac{n}{2}-\dfrac{n}{2}+\dfrac{1}{2}\sum_{\nu=1}^{n}\dfrac{|w|^2-|z_\nu|^2}{|w-z_\nu|^2}>0.
\end{align*}
This further implies that $\psi(w)\neq 0.$ Hence, $\psi(z)$ cannot have a zero in $|z|>r.$ Therefore, we conclude that all the zeros of the polynomial $zf^\prime(z)-\alpha f(z)$ lie in $|z|\leq r.$
\end{proof}
\begin{lemma}\label{l3}
 Let a polynomial $F(z)$ of degree $n$ has all its zeros in $|z|\leq 1$ and $P(z)$ be a polynomial of degree at most $n$ such that $
|P(z)|\leq |F(z)|$ for $ |z|=1,$ then for every $\alpha,\gamma\in\mathbb{C}$ with $\Re(\alpha)\leq n/2,$ $\Re(\gamma)\leq n/2$ and $|z|\geq 1,$ we have
\begin{align}\label{el1}
|zP^\prime(z)-\alpha P(z)|\leq |zF^\prime(z)-\alpha F(z)|
\end{align}
and
\begin{align}\label{el3}\nonumber
|z^2P^{\prime\prime}(z)+(1-&\alpha-\gamma)zP^\prime(z)+\alpha\gamma P(z)|\\&\leq |z^2F^{\prime\prime}(z)+(1-\alpha-\gamma)zF^\prime(z)+\alpha\gamma F(z)|.
\end{align} 
The bound is sharp and equality holds for some point z in $\vert z\vert>1$ if and only if $P(z)=e^{i\delta}F(z)$ for some $\delta\in\mathbb{R}$.
 \end{lemma}
\begin{proof}
Since the result holds trivially true, if $P(z)=e^{i\delta}F(z)$ for some $\delta\in\mathbb{R}.$ Therefore, let $P(z)\neq e^{i\delta}F(z).$  Consider the function $\phi(z)=P^*(z)/F^*(z) $ where $F^*(z)=z^n \overline{F(1/\overline{z})}$ and $P^*(z)=z^n \overline{P(1/\overline{z})}.$ Since $F(z)$ has its all zeros in $\vert z\vert\leq1,$ then $F^*(z)$ has no zero in $|z|<1.$ 
This implies that the rational function $\phi(z)$ is analytic for $\vert z\vert\leq 1.$  Also, $\vert P(z)\vert=\vert P^*(z)\vert$ and $\vert F(z)\vert=\vert F^*(z)\vert$ for $\vert z\vert=1,$ therefore, $|\phi(z)|\leq 1$ for $|z|=1.$  By invoking Maximum modulus theorem, we obtain;
\begin{align*}
\vert \phi(z)\vert<1 \qquad \text{for}\quad \vert z\vert< 1.
\end{align*} 
 On replacing $z$ by $1/z$ in the above inequality, we get $|P(z)|<| F(z)|$ for $\vert z\vert >1.$ It follows by Rouche's theorem that for any $\lambda\in\mathbb{C}$ with $|\lambda|\geq 1,$ the polynomial $P(z)-\lambda F(z)$ of degree $n$ has all its zeros in $|z|\leq 1.$ Applying Lemma \ref{l2} to the polynomial  $G(z)=P(z)-\lambda F(z),$ for $\alpha\in\mathbb{C}$ with $\Re(\alpha)\leq n/2$,  we obtain that the polynomial
\begin{align*}
zG^\prime(z)-\alpha G(z)&=z(P^\prime(z)-\lambda F^\prime(z))-\alpha (P(z)-\lambda F(z))\\& =(zP^\prime(z)-\alpha P(z))-\lambda(zF^\prime(z)-\alpha F(z))
\end{align*}
has all zeros in $\vert z\vert\leq1.$
This implies
\begin{align}\label{pr1}
\vert zP^\prime(z)-\alpha P(z)\vert \leq \vert zF^\prime(z)-\alpha F(z)\vert\qquad\text{ for} \quad \vert z\vert>1.
\end{align}
 If inequality \eqref{pr1} were not true, then there is some point $w\in\mathbb{C}$ with $\vert w\vert>1$ such that
 $\vert wP^\prime(w)-\alpha P(w)\vert > \vert wF^\prime(w)-\alpha F(w)\vert.$  By Lemma \ref{l1}, $wF^\prime(w)-\alpha F(w)\neq 0.$ Now, choose $\lambda= -\frac{ wP^\prime(w)-\alpha P(w)}{ wF^\prime(w)-\alpha F(w)}$ and note that $\lambda$ is a well defined complex number with modulus greater than 1. So, with this choice of $\lambda$, one can easily observe that $w$ is a zero of $zG^\prime(z)-\alpha G(z)$ of modulus greater than one.
 This is a contradiction, since all the zeros of this polynomial lie in $|z|\leq 1.$
 Hence, the inequality \eqref{pr1} is true, by continuity \eqref{pr1} also holds for $|z|=1.$ This proves the inequality \eqref{el1}.
 
 Finally, if we take $H(z)=zP^\prime(z)-\alpha P(z)$ and $K(z)=zF^\prime(z)-\alpha F(z)$ with $\Re(\alpha)\leq n/2,$ then by inequality \eqref{el1}, $|H(z)|\leq |K(z)|$ for $|z|=1.$ Therefore, by using inequality \eqref{el1} again, for $\gamma \in\mathbb{C}$ with $\Re(\gamma)\leq n/2,$ we get, $|zH^\prime(z)-\gamma H(z)|\leq |zK^\prime(z)-\gamma K(z)|$ for $|z|\geq 1,$ which is equivalent to \eqref{el3}. This completes the proof.
 \end{proof} 
 
 The next lemma follows immediately from lemma \ref{l3} by taking $F(z)=Q(z)$ where $Q(z)=z^n\overline{ P(1/\overline{z})}$.
 \begin{lemma} \label{14}
If $P\in\mathcal{P}_n $ and $ P(z) $ does not vanish in $|z|<1$, then for every $\alpha, \gamma\in\mathbb{C}$ with $\Re(\alpha)\leq n/2$ and $\Re(\gamma)\leq n/2$
\begin{equation}\label{tel}
|zP^\prime(z)-\alpha P(z)|\leq |zQ^\prime(z)-\alpha Q(z)|,
\end{equation}
and 
\begin{align}\nonumber\label{sel}
|z^2P^{\prime\prime}(z)+(1-&\alpha-\gamma)zP^\prime(z)+\alpha\gamma P(z)|\\&\leq |z^2Q^{\prime\prime}(z)+(1-\alpha-\gamma)zQ^\prime(z)+\alpha\gamma Q(z)|.
\end{align}
where $Q(z)=z^n\overline{ P(1/\overline{z})}$
\end{lemma}
\begin{lemma} \label{15}
If $P\in\mathcal{P}_n $ and $ P(z) $ does not vanish in $|z|<1.$  Let $Q(z)=z^n\overline{ P(1/\overline{z})}$, then for every $\alpha, \gamma\in\mathbb{C}$ with $\Re(\alpha)\leq n/2$, $\Re(\gamma)\leq n/2$ and $\beta\in \mathbb{R},$ 
\begin{equation*}
\left(zP^\prime (z)-\alpha P(z)\right)e^{i\beta}+z^n\overline{ M(1/\overline{z})}\neq 0 
\end{equation*}
and
\begin{align*}
\left(z^2P^{\prime\prime}(z)+(1-\alpha-\gamma)zP^\prime(z)+\alpha\gamma P(z)\right)e^{i\beta}+z^n\overline{ N(1/\overline{z})}\neq 0 
\end{align*}
for $|z|<1$ , where $M(z)=zQ^\prime (z)-\alpha Q(z)$, $N(z)=z^2Q^{\prime\prime}(z)+(1-\alpha-\gamma)zQ^\prime(z)+\alpha\gamma Q(z)$.
\end{lemma}
\begin{proof}
Since $P(z)=\sum_{j=0}^n a_jz^j$  does not vanish in $|z|<1$, then by lemma \ref{14} for every $\alpha\in\mathbb{C}$ with $\Re(\alpha)\leq n/2$ and $|z|=1,$ we have
\begin{align}\nonumber\label{le5}
\left|zP^\prime (z)-\alpha P(z) \right|&\leq \left|zQ^\prime (z)-\alpha Q(z) \right|\\
 \nonumber &=|M(z)|\\
  &=|z^n\overline{ M(1/\overline{z})}|.
\end{align}
Also since $P(0)\neq 0$ then deg$Q(z)=n$. Moreover, $Q(z)$ does not vanish for $|z|>1$ therefore by lemma \ref{l2} the polynomial $M(z)$ also does not vanish for $|z|>1$.This implies that $z^n\overline{ M(1/\overline{z})}\neq 0$ for $|z|<1$. Hence, by lemma \ref{l1} for $|z|<1,$ we have
\begin{center}
$\left(zP^\prime (z)-\alpha P(z)\right)e^{i\beta}+z^n\overline{ M(1/\overline{z})}\neq 0 $.  
\end{center}
The second part of lemma follows similarly by using inequality \eqref{sel} instead of \eqref{tel} of lemma \ref{14}.
\end{proof}
Let $P(z)=\sum_{j=0}^{n}a_jz^j \in\mathcal{P}_n$ and $\delta=(\delta_0,\delta_1,\ldots,\delta_n)\in \mathbb{C}^{n+1},$ then $\Lambda_\delta P(z)=\sum_{j=0}^{n}\delta_j a_jz^j$ defines a linear operator on the space of polynomials of degree at most $n.$  The operator $\Lambda_{\delta}$ is said to be admissible if it preserves one of the following properties:
\begin{enumerate}
\item[(i)] $P(z)$ has all its zeros in $\{z\in\mathbb{C}:|z|\leq 1\}$,
\item[(ii)] $P(z)$ has all its zeros in $\{z\in\mathbb{C}:|z|\geq 1\}.$
\end{enumerate}
The next lemma is due to Arestov \cite{ART}.
\begin{lemma}\label{16} Let $\phi(x)=\psi(\log x)$ where $\psi$ is a convex non-decreasing function on $\mathbb{R}$. Then for all $P\in\mathcal{P}_n$ and each admissible operator $\Lambda_\delta$, $$\int\limits_{0}^{2\pi}\phi(|\Lambda_\delta P(e^{i \theta})|)d\theta \leq \int\limits_{0}^{2\pi}\phi(c(\delta) |P(e^{i \theta})|)d\theta,$$ where $c(\delta)=\max (|\delta_0|,|\delta_n|).$
\end{lemma}
  From lemma \ref{16}, we deduce the following result:
  \begin{lemma}\label{l7}
  If $P\in\mathcal{P}_n $ and $ P(z) $ does not vanish in $|z|<1$ and $Q(z)=z^n\overline{ P(1/\overline{z})}$, then for every $\alpha,\gamma\in\mathbb{C}$ with $\Re(\alpha)\leq n/2$, $\Re(\gamma)\leq n/2$, $\beta$,$\delta$ real and $p>0$
 \begin{align}\nonumber\label{fl1}
  \int_0^{2\pi}\left|\left(e^{i\theta}P^\prime (e^{i\theta})-
  \alpha P(e^{i\theta})\right)e^{i\beta}+e^{in\theta}\overline{ M(e^{i\theta })}\right|^pd\theta\qquad\\
   \leq \left|\left(n-\alpha\right)e^{i\beta}-\overline{\alpha}\right|^p\int\limits_{0}^{2\pi}|P(e^{i\theta})|^pd\theta
     \end{align}
     and
     \begin{align}\nonumber\label{al2}
     \int_0^{2\pi}&\left|\left(e^{i2\theta}P^{\prime\prime}(e^i\theta)+(1-\alpha-\gamma)e^{i\theta} P^\prime(e^{i\theta})+\alpha\gamma P(e^{i\theta})\right)e^{i\delta}+e^{in\theta}\overline{ N(e^{i\theta})}\right|^pd\theta\\
  &\qquad\qquad \leq \left|\left(n(n-\alpha-\gamma)+\alpha\gamma\right)e^{i\delta}+\overline{\alpha\gamma}\right|^p\int\limits_{0}^{2\pi}|P(e^{i\theta})|^pd\theta,
     \end{align}
   where $M(z)=zQ^\prime(z)-\alpha Q(z)$, $N(z)=z^2Q^{\prime\prime}(z)+(1-\alpha-\gamma)zQ^\prime(z)+\alpha\gamma Q(z)$.
  \end{lemma}
  \begin{proof}
  Let$P(z)=\sum_{j=0}^{n}a_jz^j$ does not vanish in $|z|<1$, then by Lemma \ref{15}, the polynomial 
\begin{align*}
\Lambda P(z)&=\left(zP^\prime (z)-\alpha P(z)\right)e^{i\beta}+z^n\overline{ M(1/\overline{z})}\\
&=\left((n-\alpha)e^{i\beta}-\overline{\alpha}\right)a_nz^n+\ldots+\left(-\alpha e^{i\beta} +n-\overline{\alpha}\right)a_0
\end{align*}
also does not vanish in $|z|<1$ for every $\beta\in\mathbb{R}$ and $\alpha\in\mathbb{C}$ with $\Re(\alpha)\leq n/2$ and consequently the operator $\Lambda$ is admissible. Applying lemma \ref{16} with $\phi(x)=x^p,$ where $p>0$, then the inequality \eqref{fl1} follows immediately for every $p>0$.

Following the similar lines and again using Lemma \ref{15}, one can conclude that the operator $\Lambda_1$ which sends $P(z)$ to 
\begin{align*}
\Lambda_1 P(z)&=\left(z^2P^{\prime\prime}(z)+(1-\alpha-\gamma)zP^\prime(z)+\alpha\gamma P(z)\right)e^{i\delta}+z^n\overline{ N(1/\overline{z})}\\
&=\left[\left(n(n-\alpha-\gamma)+\alpha\gamma\right)e^{i\delta}+\overline{\alpha\gamma}\right]a_nz^n+\dots+\left[\alpha\gamma e^{i\delta}+n\overline{(n-\alpha-\gamma)}\right]a_0
\end{align*}
is also an admissible operator for every $\alpha, \gamma\in\mathbb{C}$ with $\Re(\alpha)\leq n/2$, $\Re(\gamma)\leq n/2$ and $\delta$ real.  Again, applying lemma \ref{16} with $\phi(x)=x^p,$ where $p>0,$ we get
\begin{align*}
\int_0^{2\pi}&\left|\left(e^{i2\theta}P^{\prime\prime}(e^{i\theta})+(1-\alpha-\gamma)e^{i\theta} P^\prime(e^{i\theta})+\alpha\gamma P(e^{i\theta})\right)e^{i\delta}+e^{in\theta}\overline{ N(e^{i\theta})}\right|^pd\theta\qquad\qquad\qquad\\
   &\qquad\leq \left|\left(n(n-\alpha-\gamma)+\alpha\gamma\right)e^{i\delta}+\overline{\alpha\gamma}\right|^p\int\limits_{0}^{2\pi}|P(e^{i\theta})|^pd\theta.
\end{align*} 
This completes the proof of lemma \ref{l7}.
  \end{proof}
 \section{Proofs of the theorems}
\begin{proof}[\bf Proof of Theorem\ref{t1}]
Let $ z_{1}, z_{2}, \ldots, z_{k} $  be the zeros of $P(z)$ which lie in $|z|>1,$ where $0\leq k\leq n,$ then all the zeros of the polynomial 
\begin{align*}
T(z)=P(z)\prod_{j=1}^{k}\dfrac{1-\bar{z_j}z}{z-z_j}
\end{align*}
lie in $|z|\leq 1.$ Moreover, $|P(z)|=|T(z)|$ for $|z|=1,$ therefore, by Lemma \ref{l3}, we have for every $\alpha\in\mathbb{C}$ with $\Re(\alpha)\leq n/2,$   
   \begin{align*}
   |zP^\prime(z)-\alpha P(z)|&\leq |zT^\prime(z)-\alpha T(z)|\qquad \text{for} \quad |z|\geq 1,
   \end{align*}
   which in particular gives for $p>0,$
\begin{equation}\label{ket}
\int\limits_{0}^{2\pi}\left|e^{i\theta}P^{\prime}(e^{i\theta})- \alpha P(e^{i\theta})\right|^pd\theta \leq \int\limits_{0}^{2\pi}\left|e^{i\theta}T^{\prime}(e^{i\theta})- \alpha T(e^{i\theta})\right|^pd\theta.
\end{equation}
Since all the zeros of polynomial $T(z)$ lies in $|z|\leq 1,$ then by lemma \ref{l2}, the polynomial
$zT^\prime(z)-\alpha T(z)$ also has all its zeros in $|z|\leq 1$ for every $\alpha\in\mathbb{C}$ with $\Re(\alpha)\leq n/2$.
Therefore, if $T(z)=c_nz^n+c_{n-1}z^{n-1}+\ldots+c_1z+c_0$, then the operator $\Lambda_{\delta}$ defined by
\begin{align*}
\Lambda_{\delta}T(z)&=zT^\prime(z)-\alpha T(z)\\
&=(n-\alpha )c_nz^n+\ldots-\alpha c_0
\end{align*}
is admissible. Hence by lemma \ref{16} with $\phi(x)=x^p$ where $p>0$, we have
\begin{equation}\label{stt}
\int\limits_{0}^{2\pi}\left|e^{i\theta}T^{\prime}(e^{i\theta})- \alpha T(e^{i\theta})\right|^pd\theta \leq (c(\delta))^p\int\limits_{0}^{2\pi}|T(e^{i\theta})|^pd\theta,
\end{equation}
where $c(\delta)=\max (|n-\alpha| ,|\alpha|)$.

\noindent Since $\Re(\alpha)\leq n/2$ then $c(\delta)=|n-\alpha|$. Thus from \eqref{stt}, we have
 \begin{equation}\label{set}
 \int\limits_{0}^{2\pi}\left|e^{i\theta}T^{\prime}(e^{i\theta})- \alpha T(e^{i\theta})\right|^pd\theta \leq |n-\alpha |^p\int\limits_{0}^{2\pi}|T(e^{i\theta})|^pd\theta.
 \end{equation}
 Combining inequalities \eqref{ket} and \eqref{set} and noting that $|T(e^{i\theta})|=|P(e^{i\theta})|$, we obtain
 \begin{align*}
 \int\limits_{0}^{2\pi}\left|e^{i\theta}P^{\prime}(e^{i\theta})- \alpha P(e^{i\theta})\right|^pd\theta \leq |n-\alpha |^p\int\limits_{0}^{2\pi}|P(e^{i\theta})|^pd\theta.
 \end{align*}
 This proves inequality \eqref{tlel} of  theorem \ref{t1} for $p>0$. The inequality \eqref{t1ec} of theorem \ref{t1} follows similarly by replacing $P(z)$ by $zP^\prime(z)-\alpha P(z)$ and $\gamma\in\mathbb{C}$ with $\Re(\gamma)\leq n/2$. To obtain these results for $p=0$, we simply make $p\rightarrow 0+$.
\end{proof}
\begin{proof}[\bf Proof of Theorem \ref{t2}]
By hypothesis $P(z)$ does not vanish in $z<1$, therefore by \eqref{tel} of lemma \ref{14} for every $\alpha\in\mathbb{C}$ with $\Re(\alpha)\leq n/2$  and $0\leq\theta\leq 2\pi,$
\begin{equation}\label{tvr}
|zP^\prime(z)-\alpha P(z)|\leq |zQ^\prime(z)-\alpha Q(z)|,
\end{equation}
where $Q(z)=z^n\overline{ P(1/\overline{z})}$.
Also by \eqref{fl1} of lemma \ref{l7},
\begin{align}\label{al1}
  \int_0^{2\pi}\left|e^{i\beta}F(\theta)+e^{in\theta}\overline{ M(e^{i\theta })}\right|^pd\theta
   \leq \left|\left(n-\alpha\right)e^{i\beta}-\overline{\alpha}\right|^p\int\limits_{0}^{2\pi}|P(e^{i\theta})|^pd\theta,
\end{align}
   where $F(\theta)= e^{i\theta}P^\prime (e^{i\theta})-
  \alpha P(e^{i\theta})$ and $M(e^{i\theta})=e^{i\theta}Q^\prime(e^{i\theta})-\alpha Q(e^{i\theta})$.
  Integrating both sides of \eqref{al1} with respect to $\beta$ from $0$ to $2\pi$, we get for each $p>0,$
  \begin{align}\label{bl1}
 \int_0^{2\pi} \int_0^{2\pi}\left|e^{i\beta}F(\theta)+e^{in\theta}\overline{ M(e^{i\theta })}\right|^pd\theta d\beta
   \leq \int_0^{2\pi}\left|\left(n-\alpha\right)e^{i\beta}-\overline{\alpha}\right|^pd\beta\int\limits_{0}^{2\pi}|P(e^{i\theta})|^pd\theta.
\end{align}
Now for every real $\beta,$ $t\geq 1$ and $p>0$, we have
\begin{center}
$\int\limits_{0}^{2\pi}|t+e^{i\beta}|^pd\beta\geq \int\limits_{0}^{2\pi}|1+e^{i\beta}|^pd\beta$.
\end{center}
If $F(\theta)\neq 0$, we take $t=|M(e^{i\theta}) /F(\theta)|$, then by \eqref{tvr}, $t\geq 1$ and we get
\begin{align*}
\int\limits_{0}^{2\pi}\left|e^{i\beta}F(\theta)+e^{in\theta}\overline{M(e^{i\theta})}\right|^pd\beta&=|F(\theta)|^p\int\limits_{0}^{2\pi}\left|e^{i\beta}+\frac{e^{in\theta}\overline{M(e^{i\theta})}}{F(\theta)}\right|^pd\beta\\
&=|F(\theta)|^p\int\limits_{0}^{2\pi}\left|e^{i\beta}+\left|\frac{M(e^{i\theta})}{F(\theta)}\right|\right|^pd\beta\\
&\geq |F(\theta)|^p\int\limits_{0}^{2\pi}|1+e^{i\beta}|^pd\beta.
\end{align*}
For $F(\theta)=0$, this inequality is trivially true. Using this in \eqref{bl1} , we conclude that for every $\alpha\in\mathbb{C}$ with $\Re(\alpha)\leq n/2,$ 
\begin{align}\label{xxx}
\int\limits_{0}^{2\pi}|F(\theta)|^pd\theta\int\limits_{0}^{2\pi}|1+e^{i\beta}|^pd\beta\leq  \int_0^{2\pi}\left|\left(n-\alpha\right)e^{i\beta}-\overline{\alpha}\right|^pd\beta\int\limits_{0}^{2\pi}|P(e^{i\theta})|^pd\theta.
\end{align}
Since \begin{align}\nonumber\label{uu}
\int_0^{2\pi}\left|\left(n-\alpha\right)e^{i\beta}-\overline{\alpha}\right|^pd\beta &\nonumber=\int_0^{2\pi}\left| |n-\alpha|e^{i\beta}+|\overline{\alpha}|\right|^pd\beta\\&=\int_0^{2\pi}\left|\left(n-\alpha\right)e^{i\beta}+\alpha\right|^pd\beta,
  \end{align}
  the desired inequality \eqref{cel} follows immediately by combining \eqref{xxx} and \eqref{uu}. The second part of theorem, that is inequality \eqref{sss}, follows similarly by using inequalities \eqref{sel} and \eqref{al2} instead of \eqref{tel} and \eqref{fl1} of lemma \ref{14} and lemma \ref{l7} respectively. This proves Theorem \ref{t2} for $p>0$. To establish theorem for $p=0$, we simply make $p\rightarrow 0$.
\end{proof}
\begin{proof}[\bf Proof of Theorem \ref{t3}]
Since $P(z)$ is a self-inversive polynomial, we have $P(z)=uQ(z)$  for all $z\in\mathbb{C}$ where $|u|=1$ and  $Q(z)=z^n\overline{ G(1/\overline{z})}$. Therefore for every $\alpha\in\mathbb{C}$ with $\Re(\alpha)\leq n/2$
\begin{center}
$ |zP^\prime(z)-\alpha P(z)|= |zQ^\prime(z)-\alpha Q(z)|$
\end{center}
for all $z\in\mathbb{C}$. Using  this equation instead of inequality \eqref{tvr} and proceeding similarly as in the proof of Theorem \ref{t2} we get the desired result. This completes the proof of Theorem \ref{t3}. 
\end{proof}
\section*{Statements and Declarations}
\begin{itemize}
\item[$\bullet$] The authors did not receive support from any organization for the submitted work.
\item[$\bullet$] The authors have no competing interests to declare that are relevant to the content of this article.
\end{itemize}
\section*{Data Availability Statement}
Data sharing not applicable to this article as no datasets were generated or analysed during the current study.

\end{document}